\documentclass[10pt,twoside]{siamltex}

\usepackage{amssymb}
\usepackage{amsmath,xcolor}
\usepackage{mathrsfs}
\usepackage{graphicx}
\usepackage{calc}%

\usepackage[mathscr]{euscript}
\newcommand{\Dt}{\Delta t}

\newcommand{\m}{\mathbf m }

\newcommand{\uu}{\mathbf u}
\newcommand{\vv}{\mathbf v}

\newcommand{\diver}{{\rm {div} \,\, }}
\newcommand{\varep}{\varepsilon}

\newcommand{\norm} [1]{\left\| {#1}\right\|}

\newcommand{\R} {\mathbb R}
\newcommand{\N} {\mathbb N}

\newcommand{\nablad} {\nabla\cdot}

\newcommand{\intb}[1]{\left\langle #1 \right\rangle}

\newcommand{\intd}[1]{\left( #1 \right)}

\newcommand{\eqdef}{\overset{\mathrm{def}}{=\joinrel=}}

\def\beq{\begin{equation}}
\def\eeq{\end{equation}} 

\def\beqs{\begin{equation*}}
\def\eeqs{\end{equation*}}

\def\bals{\begin{align*}}
\def\eals{\end{align*}}

\def\bspl{\begin{split}}
\def\espl{\end{split}}

\def\myclearpage{}
\title{Existence of a solution for generalized Forchheimer flow in porous media}
\author{Thinh Kieu \footnotemark[2]  }

\begin{document}

\maketitle
 
\renewcommand{\thefootnote}{\fnsymbol{footnote}}
\footnotetext[2]{Department of Mathematics, University of North Georgia, Gainesville Campus, 3820 Mundy Mill Rd., Oakwood, GA 30566, U.S.A. ({\tt thinh.kieu@ung.edu}).}               
               
\begin{abstract}  This paper is focused on the generalized Forchheimer flows for slightly compressible fluids. We prove the existence and uniqueness of the differential system for stationary problem. The technique of semi-discretization in time is used to prove the existence of solution for the transient problem. 
 \end{abstract}
            
            
            
\begin{keywords}
Porous media, immersible flow, generalized Forchheimer equations, existence.
\end{keywords}

\begin{AMS}
35Q35,  35D30, 35K55, 76S05.
\end{AMS}

\pagestyle{myheadings}
\thispagestyle{plain}
\markboth{Thinh Kieu}{Galerkin method for Forchheimer-Ward equation for slightly compressible fluids}
            
            
\myclearpage    
\section {Introduction} \label{Intrsec}
 We consider a fluid in porous medium  occupying a bounded domain $\Omega\subset \R^d,$ $d\ge 2$ with boundary $\Gamma$. Let $x\in\R^d $, $0<T<\infty$ and $t\in (0,T]$ be the spatial and time variables respectively. The fluid flow has velocity $v(x,t)\in \R^d$, pressure $p(x,t)\in\R$ and density $\rho(x,t)\in \R_+$.   

The  Darcy--Forchheimer equation is studied in \cite{ABHI1,HI1,HI2} of the form 
\beq\label{gF}
-\nabla p =\sum_{i=0}^N a_i |v|^{\alpha_i}v. 
\eeq  
These equations are analyzed numerically in \cite{Doug1993,EJP05,K1},
theoretically in \cite{ABHI1,HI2,HIKS1,HKP1,HK1,HK2} for single phase flows, and also in \cite{HIK1,HIK2} for two phase flows.

In order to take into account the presence of density in generalized Forchheimer equation, we modify \eqref{gF} using dimension analysis by Muskat \cite{Muskatbook} and Ward \cite{Ward64}. They proposed the following equation for both laminar and turbulent flows in porous media:
\beq\label{W}
-\nabla p =G(v^\alpha \kappa^{\frac {\alpha-3} 2} \rho^{\alpha-1} \mu^{2-\alpha}),\text{ where  $G$ is a function of one variable.}
\eeq 
In particular, when $\alpha=1,2$, Ward \cite{Ward64} established from experimental data that
\beq\label{FW} 
-\nabla p=\frac{\mu}{\kappa} v+c_F\frac{\rho}{\sqrt \kappa}|v|v,\quad \text{where }c_F>0.
\eeq

Combining  \eqref{gF} with the suggestive form \eqref{W} for the dependence on $\rho$ and $v$, we propose the following equation 
 \beq\label{FM}
-\nabla p= \sum_{i=0}^N a_i \rho^{\alpha_i} |v|^{\alpha_i} v,
 \eeq
where $N\ge 1,\alpha_0=0<\alpha_1<\ldots<\alpha_N$ are fixed real numbers, the coefficients $a_0(x,t), \ldots, a_N(x,t)$ are non-negative with 
$$ 0<\underbar a <a_0(x,t), a_N(x,t)<\bar a<\infty, \quad 0\le a_i(x,t) \le \bar a<\infty,\,  i=1,\ldots, N-1.$$ 
  

Multiplying both sides of previous equation to $\rho$, we obtain 
 \beq\label{eq1}
  \left(\sum_{i=0}^N a_i |\rho v|^{\alpha_i}\right) \rho v   =-\rho\nabla p,
 \eeq
 Denote the function $F:\Omega\times[0,T]\times\mathbb{R}^+\rightarrow\mathbb{R}^+$ a generalized polynomial with non-negative coefficients by
\beq\label{eq2}
F(x,t, z)=a_0(x,t)z^{\alpha_0} + a_1(x,t)z^{\alpha_1}+\cdots +a_N(x,t)z^{\alpha_N},\quad z\ge 0. 
\eeq 
The equation \eqref{eq1} can rewrite as 
\beq\label{eq1a} 
F(x,t, |\rho v|)\rho v = -\rho\nabla p.
\eeq 

Under isothermal condition the state equation relates the density $\rho$ with the pressure $p$ only, i.e., $\rho=\rho(p)$. Therefore, the equation of state which, for slightly compressible fluids, is
\beq
\frac 1\rho \frac{d\rho}{dp}=\frac 1\kappa=const.>0.
\eeq
Hence
\beq \label{eq3}
\nabla \rho = \frac{1}{\kappa} \rho\nabla p, 
\quad \text{ or }\quad  \rho \nabla p=\kappa\nabla \rho.
\eeq
Combining \eqref{eq1a} and \eqref{eq3} implies that  
 \beq\label{ru}
 F(x,t, |\rho v|) \rho v   =-\kappa\nabla \rho.
 \eeq

The continuity equation is
\beq\label{con-law}
\phi(x)\rho_t+{\rm div }(\rho v)=f(x,t).
\eeq
where $\phi$ is the porosity, $f$ is external mass flow rate . 

By combining \eqref{ru} and \eqref{con-law} we have
\beqs
\begin{aligned}
F(x,t,|\m|) \m = -\kappa\nabla \rho, \\
\phi(x)\rho_t+ \diver{\m}=f(x,t),
\end{aligned}
\eeqs
where $\m=\rho v$.
By rescal the variable $\rho\to\kappa\rho$, $\phi(x)\to \kappa^{-1}\phi(x) $. We can assume $\kappa=1$ to obtain system of equations  
\beq\label{main-sys}
\begin{aligned}
F(x,t,|\m|) \m = -\nabla \rho, \\
\phi(x)\rho_t+ \diver{\m}=f(x,t).
\end{aligned}
\eeq  
The Darcy- Forchheimer equation in \eqref{main-sys} can be resolve to give 
\beq\label{rua} 
\m = - K(x,t, |\nabla \rho |)\nabla \rho,
\eeq
where the function $K: \Omega\times[0,T]\times\mathbb{R}^+\rightarrow\mathbb{R}^+$ is defined for $\xi\ge 0$ by
\beq\label{Kdef}
K(x,t, \xi)=\frac{1}{F(x,t,s(x,t, \xi))},
\eeq
  with   $s=s(x,t, \xi)$  being the unique non-negative solution of  $sF(s)=\xi$.

Substituting \eqref{rua} into the second equation of \eqref{main-sys} we obtain a scalar partial differential equation (PDE) for
the density:
\beq\label{mainEq}
\phi(x)\rho_t- \diver{ (K(x,t, |\nabla \rho |)\nabla \rho) }=f(x,t), \quad (x,t)\in\Omega\times[0,T]. 
\eeq
This equations with Dirichlet boundary conditions and appropriate initial conditions was studied theoretically in \cite{HIKS1} for constants physical parameters. Using the theory of monotone operators \cite{MR0259693,s97,z90}, they proved the global existence of weak solutions.

\subsection{Notations and preliminary results}

Suppose that $\Omega$ is an open, bounded subset of $\mathbb{R}^d$, with $d=2,3,\ldots$, and has $C^1$-boundary $\Gamma=\partial \Omega$. Let $L^2(\Omega)$ be the set of square integrable functions on $\Omega$ and $( L^2(\Omega))^d$ the space of $d$-dimensional vectors which have all components in $L^2(\Omega)$.  We denote $(\cdot, \cdot)$ the inner product in either $L^2(\Omega)$ or $(L^2(\Omega))^d$ that is
$
( \xi,\eta )=\int_\Omega \xi\eta dx$  or $(\boldsymbol{\xi},\boldsymbol \eta )=\int_\Omega \boldsymbol{\xi}\cdot \boldsymbol{\eta} dx. 
$
The notation $\intb{\cdot ,\cdot}$ will be used for the $L^2(\partial \Omega)$ inner-product and $ \norm{u}_{L^p}=\norm{u}_{L^p(\Omega)}$ for standard  Lebesgue norm of the measurable function.    
 The notation $\norm {\cdot}$ will means scalar norm $\norm{\cdot}_{L^2(\Omega)}$ or vector norm $\norm{\cdot}_{(L^2(\Omega))^d}$. 



Throughout this paper, we use short hand notations, 
\beqs
\norm{\rho(t)} = \norm{ \rho(\cdot, t)}_{L^2(\Omega)}, \forall t\in (0,T) \quad \text{
 and } \quad \rho^0(\cdot) =  \rho(\cdot,0).
 \eeqs
 
  Our calculations frequently use the following exponents
\beq\label{a-const }
   s= \alpha_N+2,\quad  \alpha=\frac{\alpha_N}{\alpha_N+1}= \frac{\deg (F)}{\deg (F)+1}, \quad s^*=2-\alpha=\frac{s}{s-1}.
  \eeq

The arguments $C, C_1,\ldots$ will represent for positive generic constants and their values  depend on exponents, coefficients of polynomial  $F$,  the spatial dimension $d$ and domain $\Omega$, independent of the initial and boundary data and time step. These constants may be different place by place.


We introduce the generalization $W(\rm{div}; \Omega)$ of $H(\rm{div}; \Omega)$, defined by
 \beqs
 W(\rm{div}; \Omega) =\left\{ \vv\in (L^s(\Omega))^d , \nabla\cdot\vv \in L^2(\Omega)\right\}. 
 \eeqs 
 and equip it with the norm
 \beq \label{Greenthm}
 \norm{\vv}_{W(\rm{div}; \Omega)} = \norm{\vv}_{L^s} + \norm{\nabla\cdot \vv}.
  \eeq
Since $W(\rm{div};\Omega)$ is a closed subspace of $(L^s(\Omega))^{d+1}$, it follows that $W(\rm{div};\Omega)$ is a reflexive Banach space; the boundary $\vv\cdot\nu|_{\partial\Omega}$ exist and belong to $W^{-1/s,s^*}(\partial\Omega)$ and we have the Green's formula 
\beq
\int_\Omega \vv\nabla\psi dx+\int_\Omega \psi \nabla\cdot \vv dx= \int_{\partial\Omega} \psi \vv\cdot\nu d\sigma
\eeq  
hold for every $\vv\in W(\rm{div}; \Omega)$ and $\psi\in (W(\rm{div}; \Omega))'$, where $1/s+1/s^* =1$ (see in \cite{FM77}, Lemma 3)
\begin{lemma} The following inequality hold for all $y',y\in \R^d$. 

(i)
\beq\label{F-cont}
\left|F(x,t,|y'|)y'- F(x,t,|y|)y \right| \le C_1 \left(1+ |y'|^{\alpha_N}+ |y|^{\alpha_N} \right)|y'-y|.
\eeq

(ii)
\beq\label{F-mono}
\left(F(x,t,|y'|)y'- F(x,t,|y|)y \right)\cdot (y'-y) \ge  C_2 \left( |y'-y|^2 + |y'-y|^s\right),
\eeq
where the constants $C_1(N,\bar a, \rm {deg} (F))>0$, and $ C_2(N,\underline{a}, \rm {deg} (F) )>0.$ 
\end{lemma}
\begin{proof}

(i)  Let $\gamma(t)=\tau y'+ (1-\tau)y, \tau\in [0,1]$ and $h(t) =F(x,t,|\gamma(\tau)|)\gamma(\tau)$.   Then 
\begin{align*}
&\left|F(x,t,|y'|)y'- F(x,t,|y|)y\right|= |h(1)-h(0)| =\left|\int_0^1 h'(\tau) d\tau\right|\\
 &\qquad=\left|\int_0^1 F(x,t,|\gamma(\tau)|)(y'-y) + F_z(x,t,|\gamma(\tau)|)\frac{\gamma(\tau)(y'-y)}{|\gamma(\tau)|}\gamma(\tau)   d\tau \right|\\
&\qquad\le |y'-y|\int_0^1 F(x,t,|\gamma(t)|) + F_z(x,t, |\gamma(\tau)|)|\gamma(\tau)|  d\tau.
\end{align*}

Note that $ F_z(x,t, |\gamma(\tau)|)|\gamma(\tau)| =  \sum_{i=0}^N a_i\alpha_i |\gamma(\tau)|^{\alpha_i} \le \alpha_N F(x,t, |\gamma(\tau)|) $
thus 
\beqs
|F(x,t,|y'|)y'- F(x,t,|y|)y| \le(1+\alpha_N) |y'-y|\int_0^1 F(x,t,|\gamma(\tau)|)  d\tau. 
\eeqs
 Using the inequality $x^\beta \le 1 +x^\gamma$ for $x\ge 0,  0<\beta<\gamma$ we find that  
 \beqs
 F(x,t,s) \le \max_{i=0,\ldots, N}a_i(x,t)\sum_{i=0}^N 1+s^{\alpha_N}\le (N+1) \max_{i=0,\ldots, N} a_i(x,t)(1+s^{\alpha_N}). 
 \eeqs
 Thus
\begin{align*}
|F(x,t, |y'|)y'- F(x,t,|y|)y| &\le(1+\alpha_N)(N+1) \max_{i=0,\ldots, N} a_i |y'-y|\Big(1+ \int_0^1 |\gamma(\tau)|^{\alpha_N}  d\tau\Big)\\
&\le (1+\alpha_N)(N+1) \max_{i=0,\ldots, N}a_i |y'-y|\Big(1+ \int_0^1 (|y'|+ |y|)^{\alpha_N}  d\tau\Big) \\
&\le 2^{\alpha_N}(1+\alpha_N)(N+1) \max_{i=0,\ldots, N} a_i \Big(1+ |y'|^{\alpha_N}+ |y|^{\alpha_N} \Big)|y'-y|,
\end{align*}
which proves \eqref{F-cont} hold. 
 
(ii)  Let $k(\tau) =F(x,t,|\gamma(\tau)|)\gamma(\tau)(y'-y)$.  Then 
\begin{align*}
&(F(x,t,|y'|)y'- F(x,t,|y|)y)\cdot(y'-y)= k(1)-k(0) =\int_0^1 k'(\tau) d\tau\\
&\qquad =\int_0^1\Big (F(x,t, |\gamma(\tau)|)|y'-y|^2 + F_z(x,t, |\gamma(\tau)|)\frac{|\gamma(\tau)(y'-y)|^2}{|\gamma(\tau)|}\Big) d\tau\\
&\qquad=|y'-y|^2\int_0^1 \Big(F(x,t,  |\gamma(\tau)|) + F_z(x,t,|\gamma(\tau)|)|\gamma(\tau)|\Big) d\tau \\
&\qquad\ge (1+\alpha_1)|y'-y|^2\int_0^1 F(x,t,|\gamma(\tau)|) d\tau\\
&\qquad\ge (1+\alpha_1)|y'-y|^2\Big(a_0+a_N\int_0^1 |\gamma(t)|^{\alpha_N} dt\Big).
\end{align*}
The two last inequality are obtained by using the inequalities 
\beqs
 F_z(x,t, |\gamma(\tau)|)|\gamma(\tau)| \ge \alpha_1 F(x,t, |\gamma(\tau)|)\quad \text { and } \quad F(x,t,|\gamma(\tau)|)\ge a_0+a_N|\gamma(\tau)|^{\alpha_N}.
\eeqs
It is proved (see e.g in \cite{CHIK1} Lemma 2.4, or \cite {MR2566733}  p.13, 14) that 
\beqs
\int_0^1 |\gamma(t)|^{\alpha_N} dt \ge \frac{|y'-y|^{\alpha_N}}{2^{\alpha_N+1}(\alpha_N+1) }. 
\eeqs
Hence
\beqs
(F(x,t,|y'|)y'- F(x,t,|y|)y)\cdot(y'-y)\ge (1+\alpha_1)|y'-y|^2\left(a_0+a_N\frac{|y'-y|^{\alpha_N}}{2^{\alpha_N+1}(\alpha_N+1) }\right).
\eeqs
\end{proof}

  The remainder of the article is organized as follows.  
  In section \S \ref{Intrsec}, we introduce the notations and the relevant results.
 In Section \ref{StatProb}  we consider the stationary problem of \eqref{main-sys}. The existence and uniqueness of a solution  is proved in Theorem~\ref{stationaryProb}. 
In Section \ref{SemiProb} we investigate the semi-discrete problem after discretization of the time-derivative in \eqref{main-sys} and show again the existence and uniqueness of a solution in Theorem~\ref{Sol-semidiscreteProb} . 
Finally, in Section \ref{TransProb}, we study the transient problem governed by~\eqref{TProb} with homogeneous boundary conditions. We derive a priori estimates of the solutions to \eqref{semidiscrete-prob}. These are used to prove the solvability of the transient problem~\eqref{TProb}.

\section {The stationary problem}\label{StatProb}
We consider the stationary problem governed by the Darcy-Forchheimer equation and the stationary continuity equation together with Dirichlet boundary condition   
\beq\label{stationaryProb}
\begin{aligned}
F(x,|\m|) \m = -\nabla \rho,\qquad x\in \Omega,\\
\diver{\m}=f(x), \qquad x\in \Omega,\\
\rho =-\rho_b(x),\qquad  x\in\partial \Omega.
\end{aligned}
\eeq
\subsection{The mixed formulation of the stationary problem}
The mixed formulation of \eqref{stationaryProb} read as follows: Find $(\m,\rho)\in W(\rm {div}; \Omega)\times L^2(\Omega)$ such that  
\beq\label{WeakStationaryProb}
\begin{aligned}
(F(x,|\m|) \m, \vv) -(  \rho, \nabla\cdot \vv)=- \langle  \rho_b, \vv\cdot \nu  \rangle, \quad \forall \vv\in W(\rm {div}; \Omega),\\
(\nabla\cdot \m, q)=(f,q), \quad\forall q\in L^2(\Omega).
\end{aligned}
\eeq

We introduce 
 a bilinear form $b:W(\rm {div}; \Omega)\times L^2(\Omega)\to \R$ and a nonlinear form $a: (L^{s}(\Omega))^d\times (L^{s}(\Omega) )^d\to \R$ by mean of 
\beqs
b(\vv,q)=  (\nabla \cdot \vv, q) \, \text{ for }  \vv\in W(\rm {div}; \Omega), q\in L^2(\Omega),
\eeqs
and
 \beqs
 a(\uu,\vv)= (F(x,t,|\uu|)\uu, \vv ) \, \text { for } \uu,\vv\in (L^{s}(\Omega))^d. 
\eeqs
Then we rewrite the mixed formulation \eqref{WeakStationaryProb} as follows:
Find $(\m,\rho)\in W(\rm div, \Omega)\times L^2(\Omega)\equiv V\times Q$ such that  
\beq\label{equivform}
\begin{aligned}
a(\m, \vv) -b(\vv,\rho) =- \intb{\rho_b, \vv\cdot \nu }, \quad \forall \vv\in W(\rm{div}; \Omega)\\
b(\m, q)=(f,q), \quad \forall q\in L^2(\Omega).
\end{aligned}
\eeq 
\subsection{Existing results} This subsection is devoted to establish  the existence and the uniqueness of weak solution of the stationary problem~\eqref{stationaryProb}.   
 \begin{theorem}\label{SolofStationaryProb}
Suppose
$f\in L^2(\Omega),$ and $\rho_b\in W^{1/s, s}  (\partial \Omega)$. The mixed formulation \eqref{WeakStationaryProb} of the stationary problem \eqref{stationaryProb} has unique solution $(\m,\rho)\in W(\rm{div};\Omega)\times L^2(\Omega)$.  
\end{theorem}
\begin{proof}
 We use regularization to show the existence of a weak solution $(\m, \rho)\in V\times Q$ to problem~\eqref{WeakStationaryProb}. The proof includes many steps.  
In step 1, we introduce an approximate problem. 
In step 2 we show that the approximate solution $(\m_\varep, \rho_\varep)$ is bounded independence of $\varep$.  
In step 3. We prove the limit $(\m, \rho)$ of the approximate solution $(\m_\varep, \rho_\varep)$ satisfy problem \eqref{WeakStationaryProb}. 
Step 4 is devoted to prove the uniqueness of weak solution $(\m, \rho)$ of problem \eqref{WeakStationaryProb}. 

{\bf Step 1. } For the fixed $\varep>0$, we consider the following regularized problem: Find $(\m_\varep,\rho_\varep)\in V \times Q$ such that  
\beq\label{reg-prob}
\begin{aligned}
a(\m_\varep, \vv)+ \varep \intd{\nabla\cdot \m_\varep , \nabla\cdot \vv}  - b(\vv,\rho_\varep) =- \intb{\rho_b,\quad \vv\cdot \nu }, \quad  \forall \vv\in V\\
\varep(\rho_\varep, q)+ b(\m_\varep, q)=(f,q), \quad  \forall q\in Q.
\end{aligned}
\eeq 
\begin{lemma}\label{StatSol}
For every $\varep>0$ there is unique solution $(\m_\varep, \rho_\varep)\in V\times Q$ of the regularization  problem \eqref{reg-prob}.  
\end{lemma}
\begin{proof}
Adding the left hand side of \eqref{reg-prob}, we obtain the nonlinear form defined on $V\times Q$, 
\beq\label{a-eps}
a_\varep((\m_\varep,\rho_\varep), (\vv,q) ) := a(\m_\varep,\vv) + \varep(\nabla\cdot \m_\varep,\nabla\cdot \vv) - b(\vv, \rho_\varep) +\varep(\rho_\varep,q)+b(\m_\varep,q), \text { for } (\vv, q)\in V\times Q. 
\eeq
A nonlinear operator $\mathcal A_\varep: (V\times Q )\to (V\times Q)'$ defined by 
\beqs
\intb{\mathcal A_\varep((\uu,p)), (\vv,q) }_{(V\times Q)'\times (V\times Q)} = a_\varep((\uu,p), (\vv,q)). 
\eeqs
Then 
$\mathcal A_\varep$ is continuous, coercive and strictly monotone. 

Applying the theorem of Browder and Minty (see in \cite{zeidler1989}, Thm. 26.A) for every $\tilde f\in (V\times Q)'$, there exists unique solution $(\m_\varep,\rho_\varep)\in V\times Q$ of the operator equation $\mathcal A_\varep (\m_\varep, \rho_\varep) = \tilde f$. In particular, we choose the linear form$ \tilde f$ defined by $\tilde f (\vv, q) :=  -\langle  \rho_b, \vv\cdot \nu  \rangle + (f,q)$, which arises by adding the right hand sides of \eqref{reg-prob}.
Therefore \eqref{reg-prob} has a unique solution.

 The rest of the proof proves $\mathcal A_\varep$ continuous, coercive and strictly monotone.
   
     For the continuity,  
\beq\label{est0}
\begin{aligned}
&\intb{\mathcal A_\varep((\uu_1,p_1)-\mathcal A_\varep((\uu_2,p_2)), (\vv,q) }_{(V\times Q)'\times (V\times Q)}\\
&\qquad=a(\uu_1,\vv)-a(\uu_2,\vv) + \varep(\nabla \uu_1-\nabla\uu_2,\nabla\vv) - b(\vv, p_1-p_2)\\
&\hspace{6cm} +\varep(p_1-p_2,q)+b(\uu_1-\uu_2,q).
\end{aligned}
\eeq
Using the \eqref{F-cont} we have 
\beq\label{est1}
\begin{split}
a(\uu_1,\vv)-a(\uu_2,\vv) &\le \intd{ (1+|\uu_1|^{s-2}+|\uu_2|^{s-2})|\uu_1-\uu_2|, |\vv|}\\
&\le  \norm{(1+|\uu_1|^{s-2}+|\uu_2|^{s-2})|\uu_1-\uu_2|}_{L^{s^*}}\norm{\vv}_{L^s}.
\end{split}
\eeq
Applying  H\"older's inequality leads to
 \beq\label{est2}
 \begin{aligned}
 \norm{(1+|\uu_1|^{s-2}+|\uu_2|^{s-2})|\uu_1-\uu_2|}_{L^{s^*}}
 &\le  \norm{1+|\uu_1|^{s-2}+|\uu_2|^{s-2}}_{0,\frac{s-2}{s}}\norm{\uu_1-\uu_2}_{L^s}\\
 & \le C \left(1+\norm{\uu_1}_{L^s}^{s-2}+ \norm{\uu_2}_{L^s}^{s-2}\right)\norm{\uu_1-\uu_2}_{L^s},
 \end{aligned}
 \eeq  
and
\beq\label{est3}
\begin{split}
&\varep(\nabla\cdot \uu_1-\nabla\cdot\uu_2,\nabla\cdot\vv) - b(\vv, p_1-p_2) +\varep(p_1-p_2,q)+b(\uu_1-\uu_2,q) \\ 
&\quad\le\varep\norm{\nabla \cdot(\uu_1 -\uu_2)}\norm{\nabla\cdot \vv}
+\norm{\nabla \cdot\vv}\norm{p_1-p_2}+\varep\norm{p_1-p_2}\norm{q}+\norm{\nabla \cdot(\uu_1 -\uu_2)}\norm{q}\\
&\quad\le (1+\varep)\left(\norm{\nabla \cdot(\uu_1 -\uu_2)}+\norm{p_1-p_2}\right)\left(\norm{\nabla\cdot \vv}+\norm{q}\right).
\end{split}
\eeq
Combining \eqref{est0}--\eqref{est3} gives
\beqs
\begin{split}
&\intb{\mathcal A_\varep((\uu_1,p_1)-\mathcal A_\varep((\uu_2,p_2), (\vv,q) )}_{(V\times Q)'\times (V\times Q)}\\
&\quad\le  C_\varep\left( 1+\norm{\uu_1}_{L^s}^{s-2}+\norm{\uu_2}_{L^s}^{s-2}\right)\left(\norm{\uu_1-\uu_2}_{L^s}+ \norm{\nabla \cdot(\uu_1 -\uu_2)}+\norm{p_1-p_2}\right)\\
&\hspace{5cm} \times       \left( \norm{\vv}_{L^s}+ \norm{\nabla\cdot \vv}+\norm{q}\right)\\
&\quad\le  C_\varep\left( 1+\norm{\uu_1}_{L^s}^{s-2}+\norm{\uu_2}_{L^s}^{s-2}\right)\left(\norm{\uu_1-\uu_2}_V+\norm{p_1-p_2}_Q\right)\left( \norm{\vv}_V+ \norm{q}_Q\right),
\end{split}
\eeqs
for all $\vv\in V, q\in Q.$
Thus
\beqs
\norm{\mathcal A_\varep((\uu_1,p_1)-\mathcal A_\varep((\uu_2,p_2)}_{(V\times Q)'}\le C_\varep( 1+\norm{\uu_1}_{L^s}^{s-2}+\norm{\uu_2}_{L^s}^{s-2})(\norm{\uu_1-\uu_2}_V+\norm{p_1-p_2}_Q).
\eeqs

For $\mathcal A_\varep$ is the coercive.     
\beqs
\begin{split}
\intb{\mathcal A_\varep(\uu,p), (\uu,p) }_{(V\times Q)'\times (V\times Q)}
&=a(\uu,\uu)+ \varep \left(\norm{\nabla\cdot\uu}^2+ \norm{p}^2 \right)\\
&\ge C \left( \norm{\uu}^2+  \norm{\uu}_{L^s}^{s}\right) +\varep \left( \norm{\nabla \cdot\uu}^{2}  +\norm{p}_Q^2\right)\\
&\ge \min\{C,\varep\}\min\{1,\norm{\uu}_{L^s}^{s-2}\} \left(\norm{\uu}_V^2+\norm{p}_Q^2\right). 
\end{split}
\eeqs  

For $\mathcal A_\varep$ is the strictly monotone. 
\beqs
\begin{split}
&\intb{\mathcal A_\varep (\uu,p)-\mathcal A_\varep(\vv,q), (\uu-\vv,p-q) }_{(V\times Q)'\times (V\times Q)}\\
&\quad =a(\uu,\uu-\vv) - a(\vv,\uu-\vv)+ \varep\left(\norm{\nabla\cdot \uu-\nabla\cdot\vv}^2 +\norm{p-q}^2\right)\\
&\quad\ge C(\norm{\uu-\vv}^2 +\norm{\uu-\vv}_{L^s}^s ) + \varep\left(\norm{\nabla\cdot \uu-\nabla\cdot\vv}^2 +\norm{p-q}^2\right)\\
 &\quad \ge \min\{C,\varep\}\min\{1, \norm{\uu-\vv}_{L^s}^{s-2} \} \norm{\uu-\vv}_V^2 +\varep\norm{p-q}_Q^2 >0,\,  \forall (\uu,p)\neq(\vv, q).    
\end{split}
\eeqs  
\end{proof}

{\bf Step 2.} 
Next, we show that the solution $(\m_\varep, \rho_\varep)$ is bounded independently of $\varep$. To do this we use the following result (see in \cite{PG16} Lemma A.3 or \cite {Sandri1993}~Lemma A.1),
\begin{lemma}
Let $s>1$ and $1/s+1/s^*=1$. Then there exists a constant $C_*>0$ such that
\beq\label{supinfcdn}
C_*\norm{q}_{L^{s^*}}\le \sup_{\vv \in W(\rm{div},\Omega)} \frac{b(\vv,q)}{\norm{\vv}_{ W(\rm{div},\Omega)} }
\eeq 
for all $\vv\in  W(\rm{div},\Omega), q\in L^s(\Omega).$
\end{lemma} 
\begin{lemma}\label{stationary-sol-indep-eps}
There exist constants $\mathcal K_1, \mathcal K_2 >0$, independent of $\varep$, such that for sufficiently small $\varep>0$ the solution $(\m_\varep,\rho_\varep)$ of \eqref{reg-prob} satisfies the following estimates
\beq
\norm{\m_\varep}_V\le \mathcal K_1 \quad \text { and } \quad \norm{\rho_\varep}_Q \le \mathcal K_2.
\eeq
\end{lemma}
\begin{proof}
We begin with a bound for the norm of $\nabla\cdot\m_\varep$. Using the second equation of \eqref{reg-prob} with $q=\nabla\cdot \m_\varep\in L^2(\Omega)$, we obtain
\beqs
\norm{\nabla\cdot \m_\varep}^{2} 
\le  \norm{f}\norm{\nabla\cdot \m_\varep} + \varep \norm{\rho_\varep}\norm{\nabla\cdot \m_\varep}.
\eeqs
Hence 
\beq\label{bound-divm}
\norm{\nabla\cdot \m_\varep} \le  \norm{f} + \varep \norm{\rho_\varep}_Q.
\eeq
Choosing the test functions $(\vv, q)=(\m_\varep, \rho_\varep)$ in \eqref{reg-prob} gives
\beq\label{estRHS1}
\begin{aligned}
a(\m_\varep, \m_\varep)+\varep(\nabla\cdot\m_\varep, \nabla\cdot\m_\varep)+ \varep(\rho_\varep, \rho_\varep)&=- \langle  \rho_b, \m_\varep\cdot \nu  \rangle + (f,\rho_\varep)  \\
&=-\intd{ \nabla \cdot \m_\varep,  \rho_b }-\intd{ \nabla\rho_b , \m_\varep}+ \intd {f,\rho_\varep} \\
&\le \norm{\rho_b}_{V'} \left(\norm{\m_\varep}_{L^s} + \norm{\nabla \cdot \m_\varep} \right)  +\norm{f}\norm{\rho_\varep} .
\end{aligned}
\eeq
Using \eqref{bound-divm} we find that
\beq\label{bound-mvarep}
C\norm{\m_\varep}_{L^s}^{s} +\varep\norm{\nabla \cdot \m_\varep}^2 +\varep\norm{\rho_\varep}_Q^2
\le  \norm{\rho_b}_{V'} \left(\norm{\m_\varep}_{L^s} + \norm{f} + \varep \norm{\rho_\varep}_Q \right)  +\norm{f}\norm{\rho_\varep}_{Q}. 
\eeq
To bound $\rho_\varep$ we employ the inf-sup condition \eqref{supinfcdn}. The first equation in \eqref{reg-prob} and the above estimate for $\varep\norm{ \nabla\cdot\m_\varep}_{L^s}^2$, we have 
\beqs
\begin{split}
C_*\norm{\rho_\varep}_Q&\le \sup_{\vv\in V} \frac{b(\vv,\rho_\varep)}{\norm{\vv}_V} 
= \sup_{\vv\in V} \frac{a(\m_\varep,\vv) +\varep(\nabla\cdot\m_\varep,\nabla\cdot\vv)+ \langle  \rho_b, \vv\cdot \nu  \rangle}{\norm{\vv}_V}\\
&\le\sup_{\vv\in V} \frac{C(\norm{\m_\varep}_{L^s}+\norm{\m_\varep}_{L^s}^{s-1} ) \norm{ \vv}_{L^s}+\varep\norm{\nabla\cdot \m_\varep}\norm{\nabla\cdot \vv} +\norm{\rho_b}_{V'} (\norm{\vv}_{L^s} + \norm{\nabla \cdot \vv} )   }{\norm{\vv}_V}\\
&\le C(\norm{\m_\varep}_{L^s}+\norm{\m_\varep}_{L^s}^{s-1} ) +\varep\norm{\nabla\cdot \m_\varep}^{2} +\norm{\rho_b}_{V' }\\
&\le C(\norm{\m_\varep}_{L^s}+\norm{\m_\varep}_{L^s}^{s-1} ) +\varep\Big( \norm{f} + \varep \norm{\rho_\varep}_Q \Big)^{2} +\norm{\rho_b}_{V' }\\
&\le C(\norm{\m_\varep}_{L^s}+\norm{\m_\varep}_{L^s}^{s-1} ) +2\varep  \norm{f}^2 + 2\varep^{3}  \norm{\rho_\varep}_Q +\norm{\rho_b}_{V' }. 
\end{split}
\eeqs
for some constant $C_*>0$. Hence, for sufficiently small $\varep$ (e.g. $\varep\le \sqrt[3]{C_*/2}$  ), 
\beq\label{bound-rhovarep}
\norm{\rho_\varep}_Q\le C\left(\norm{\m_\varep}_{L^s}+\norm{\m_\varep}_{L^s}^{s-1}  +\norm{f}^2 +  \norm{\rho_b}_{V' }\right). 
\eeq  
Substituting \eqref{bound-rhovarep} into \eqref{bound-mvarep}, we find that 
\begin{align*}
 \norm{\m_\varep}_{L^s}^{s} 
&\le C\norm{\rho_b}_{V'} \left(\norm{\m_\varep}_{L^s} +\norm{f} \right)\\
&\quad+ C\left(\norm{\rho_b}_{V'}+\norm{f}\right)
\left( \norm{\m_\varep}_{L^s}^{s-1}+\norm{\m_\varep}_{L^s}  +\norm{f}^2 +  \norm{\rho_b}_{V' }\right).
\end{align*}
Then by using Young's inequality we obtain
\beq
 \norm{\m_\varep}_{L^s}^{s} \le C K_1,
\eeq
where 
\beq\label{K1def}
K_1= \left(\norm{\rho_b}_{V'} + \norm{f}\right)\left(\norm{\rho_b}_{V'} + \norm{f}^2\right) +\norm{\rho_b}_{V'}^{s} + \norm{f}^{s}+1.
\eeq

Insert this into \eqref{bound-rhovarep} yields 
\beqs
\norm{\rho_\varep}_Q\le C\mathcal K_2, 
\eeqs
where
$
\mathcal K_2=K_1^{1/s}+K_1^{(s-1)/s}  +\norm{f}^2 +  \norm{\rho_b}_{V'}. 
$
Using this estimate in \eqref{bound-divm} yields 
\beqs
\norm{\nabla\cdot \m_\varep} \le  \norm{f} +  C\mathcal K_2 \le C(K_1+\mathcal K_2).
\eeqs
Therefore $\norm{\m_\varep}_V\le C\mathcal K_1$ where $\mathcal K_1 =K_1+\mathcal K_2$ independence of $\varep$.
\end{proof}

{\bf Step 3.} Adding the left hand side of \eqref{WeakStationaryProb} we obtain the following nonlinear form defined on $V\times Q$ by
\beqs
a( (\m, \rho), (\vv,q)  ):= a(\m, \vv) - b(\vv,\rho) + b(\m, q).
\eeqs
Consider the nonlinear operator $\mathcal A: V\times Q \to (V\times Q)'$ defined by 
\beqs
\intb{\mathcal A(\uu,p), (\vv,q)}_{ (V\times Q)' \times (V\times Q) }:=a( (\uu, p), (\vv,q)  ).
\eeqs
Set $\varep=1/n$, and let $(\m_n, \rho_n)$ be the unique solution of the regularized problem \eqref{reg-prob}. Since $(\m_n,\rho_n)$ is bounded sequence in $V\times Q,$ there exist a weakly convergent subsequence, again denoted by $(\m_n,\rho_n)$, with weak limit $(\m,\rho)\in V\times Q.$ 
\beq
\begin{split}
\norm{\mathcal A(\m_n,\rho_n) -\tilde f}_{(V\times Q)'} &= \sup_{(\vv,q)\neq {\bf 0} } \frac{| a((\m_n,\rho_n), (\vv,q))-\tilde f(\vv,q) |}{ \norm{(\vv,q)}_{V\times Q} }\\
&=\sup_{(\vv,q)\neq {\bf 0} } \frac{| a(\m_n,\vv) -b(\vv, \rho_n)+b(\m_n,q) - \tilde f(\vv,q) |}{ \norm{(\vv,q)}_{V\times Q} }.
\end{split}
\eeq 
Noting from \eqref{reg-prob} that
\begin{align*}
\left|a(\m_n,\vv) -b(\vv, \rho_n)+b(\m_n,q) - \tilde f(\vv,q)\right|
&=\frac 1 n \left|\intd{ \nabla\cdot\m_n,\nabla\cdot\vv}+ \intd{\rho_n,q} \right|\\
&\le \frac 1 n \left(\norm{\nabla\cdot\m_n}\norm{\nabla\cdot\vv} + \norm{\rho_n}\norm{q} \right)\\
&\le \frac 1 n \left(\norm{\nabla\cdot\m_n} + \norm{\rho_n} \right)\left( \norm{\vv}_V + \norm{q}_Q \right)\\
&= \frac 1 n \left(\norm{\nabla\cdot\m_n} + \norm{\rho_n}_{Q} \right)\norm{(\vv,q)}_{V\times Q} .
\end{align*}
Thus 
\beq
\begin{split}
\norm{\mathcal A(\m_n,\rho_n) -\tilde f}_{(V\times Q)'} 
\le \frac{C}{ n}\left( \norm{\nabla\cdot\m_n} +\norm{\rho_n}_Q  \right)\overset{n\to\infty}{\longrightarrow} 0.
\end{split}
\eeq 

The sequence $\mathcal A(\m_n,\rho_n)$ converges strongly in $(V\times Q)'$ to $\tilde f$ defined by $\tilde f(\vv,q) :=- \langle  \rho_b, \vv\cdot \nu  \rangle +(f,q).$ Thus we can
conclude that $\mathcal A(\m,\rho) = \tilde f$ in $(V\times Q)' $ (see e.g. \cite{z90}, p. 474), i.e., $(\m,\rho)$ is a solution of problem~\eqref{equivform}.

{\bf Step 4.} To show the uniqueness we consider two solutions $(\m_1, \rho_1)$ and  $(\m_2,\rho_2)$ of \eqref{WeakStationaryProb}. Using the test function $\vv=\m_1-\m_2,$ and $q=\rho_1-\rho_2$ we obtain 
\beq
\begin{aligned}
a(\m_1,\m_1-\m_2)-a(\m_2,\m_1-\m_2)  -\left( b(\m_1-\m_2,\rho_1)- b(\m_1-\m_2,\rho_2)  \right) = 0, \\
b(\m_1, \rho_1-\rho_2)-b(\m_2, \rho_1-\rho_2)=0.
\end{aligned}
\eeq
Adding these equations yield 
\beqs
\begin{split}
0=a(\m_1,\m_1-\m_2)-a(\m_2,\m_1-\m_2) 
\ge C_2 \left(\norm{ \m_1-\m_2}^2+\norm{ \m_1-\m_2}_{L^s}^s  \right) .
\end{split} 
\eeqs 
It follows that $\m_1=\m_2$. If $\m\in V$ is given then $\rho\in L^2(\Omega)$ is defined as a solution of the variational equation $b(\vv, \rho)= g(\vv)+a(\m,\vv)   $ for all $\vv\in V$. Therefore the uniqueness of $\rho$ is directly consequence of the injective of the operator $B': L^2(\Omega) \to V'$ ( see in \cite{BF91} \S II, Thm. 1.6).
\end{proof}



\section {The semi-discrete problem}\label{SemiProb}

We return to the transient problem governed by \eqref{main-sys}. We discretize \eqref{main-sys} in time using the implicit Euler method. This yields not only a method to solve the transient problem numerically, but also an approach to prove its solvability, the technique of semi-discretization. We define a partition $0 = t_0 < t_1 < . . . < t_J = T$ of the segment $(0, T)$ into $J$ intervals of constant length $\Delta t = T/J$, i.e., $t_j = j\Delta t$ for $j = 0, \ldots ,J$. In the following for $j = 0, \ldots ,J$ we use the denotations
$\rho^j := \rho(\cdot, jt)$ and $\m^j := \m(\cdot, jt)$ for the unknown solutions and, analogously defined, $ \rho_b^j$ for the boundary conditions and $f^j$ for the source term. 
\beq\label{semidiscreteProb}
\begin{aligned}
\left(\sum_{i=0}^N a_i^j|\m^j|^{\alpha_i}\right) \m^j = -\nabla \rho^j,  \quad x\in\Omega, \\
\phi\frac{\rho^j -\rho^{j-1}}{\Delta t} +\nablad \m^j=f^j, \quad x\in\Omega,\\
\rho = -\rho_b^j,  \quad x\in\partial\Omega,\\
\rho(x,0)=\rho_0(x), \quad x\in\Omega.
\end{aligned}
\eeq
For each $j\in \{1,\cdots, J\}$ we make the following assumptions
\begin{itemize}
\item [$H_1$.] $0<\underline{\phi}\le \phi(x) \le \overline\phi<\infty.$
\item [$H_2$.] $f^j \in L^2(\Omega)$ .
\item [$H_3$.] $\rho_b^j\in W^{1/s, s}  (\partial \Omega)$,  $\rho_0\in W_0^{1,s^*}(\Omega)\cap L^2(\Omega)$.
\item [$H_4$.] $a_i^j(x)\in L^\infty(\Omega), i=0,\ldots,N$ .
\end{itemize}
{\bf Mixed formulation of the semi-discrete problem.} 
The discretization in time of the continuity equation \eqref{semidiscreteProb}  with the implicit Euler method yields for each $j \in \{1, . . . ,J\}.$ Then 
  $\{\m^j,\rho^j\}  \in V\times Q$ such that
\beq\label{w1}
\begin{aligned}
\displaystyle \intd{ \big(\sum_{i=0}^N a_i^j|\m^j|^{\alpha_i}\big)\m^j ,\vv } - \intd{ \rho^j,\nabla \vv} =-\langle  \rho^j_b, \vv\cdot \nu  \rangle,   \quad \vv\in V,\\
\displaystyle \intd{ \frac{\phi\rho^j }{\Delta t},q}+\intd{\nabla \cdot\m^j, q} =  \intd{ f^j, q }+\intd{ \frac{\phi\rho^{j-1}}{\Delta t},q },  \quad q\in Q,
\end{aligned}
\eeq
with $\rho^0 = \rho_0(x).$
Using $a$ and $b$ defined in Section~\ref{StatProb}, we write the mixed formulation \eqref{w1} in the following way.
Find $(\m^j, \rho^j)\in V \times Q$, such that
\beq\label{eqatstepj}
\begin{aligned}
a(\m^j, \vv) - b(\vv, \rho^j) = -\langle  \rho^j_b, \vv\cdot \nu  \rangle, \quad \forall \vv\in V,\\
\intd{ \frac{\phi\rho^j }{\Delta t},q} + b(\m^j,q) =\intd { \bar f^j ,q }, \quad\forall q\in Q, 
\end{aligned}
\eeq
where 
$\bar f^j = f^j +\frac{\phi}{\Delta t}\rho^{j-1}.
$

The remainder of this section we restrict our considerations problem \eqref{eqatstepj} to a fixed time step $j$. For simplicity, we omit the superscript $j$.
\subsection{Regularization of the semi-discrete problem} 
We use the technique of regularization again.  For the fixed $\varep >$ 0, we consider the following regularized problem. Find $(\m_\varep, \rho_\varep) \in V \times Q $ such that
\beq\label{reg-eqatstepj}
\begin{aligned}
a(\m_\varep, \vv)+\varep(\nablad\m_\varep,\nablad\vv ) - b(\vv, \rho_\varep) = -\langle  \rho_b, \vv\cdot \nu  \rangle,&\quad  \forall \vv\in V,\\
\intd{ \frac{\phi}{\Delta t}\rho_\varep,q}+ b(\m_\varep,q) =\intd { \bar f ,q },&\quad \forall q\in Q.
\end{aligned}
\eeq
In the same manner as Lemma~\ref{StatSol} we obtain
\begin{lemma}
For every $\varep$ there exists a unique solution $(\m_\varep, \rho_\varep) \in V \times Q$ of the regularized semidiscrete problem \eqref{reg-eqatstepj}.
\end{lemma}

Next, we show that the solution $(\m_\varep, \rho_\varep)$ of \eqref{reg-eqatstepj} is bounded independently of $\varep$.
\begin{lemma}\label{boundedness-discrete-sol}
There exist constants $\mathcal K_1, \mathcal K_2 >0$, independent of $\varep$, such that for sufficiently small $\varep>0$ the solution $(\m_\varep,\rho_\varep)$ of \eqref{reg-eqatstepj} satisfies the following estimates:
\beq\label{discrete-bound}
\norm{\m_\varep}_V\le \mathcal K_1 \quad \text { and } \quad \norm{\rho_\varep}_Q \le \mathcal K_2.
\eeq
\end{lemma}
\begin{proof} 
As in the proof of Lemma~\ref{stationary-sol-indep-eps} we begin with an estimate for the norm of $\nabla\cdot\m_\varep$.
Using the second equation of \eqref{reg-eqatstepj} with $q= \nabla\cdot \m_\varep\in L^2(\Omega)$ we obtain
\beq\label{divm-eps}
\norm{\nabla\cdot \m_\varep} \le  \norm{\bar f} + \frac{\bar\phi}{\Delta t} \norm{\rho_\varep}_Q.
\eeq
The estimation of $\norm{\m_\varep}_{L^s}$ is based on choosing the test functions $(\vv,q)=(\m_\varep, \rho_\varep)$ in \eqref{reg-eqatstepj}. Then we obtain the estimate 
\beq\label{asd}
a(\m_\varep, \m_\varep)+\varep(\nablad\m_\varep,\nablad\m_\varep)+\intd{ \frac{\phi\rho_\varep }{\Delta t},\rho_\varep}
\le \norm{\rho_b}_{V'} \left(\norm{\m_\varep}_{L^s} + \norm{\nabla \cdot \m_\varep} \right)  +\norm{\bar f}\norm{\rho_\varep} .
\eeq
Thanks to the monotonicity of $F$, estimate \eqref{divm-eps}. It follows from \eqref{asd}  that
\begin{multline*}
C(\norm{\m_\varep}_{L^s}^{s}+ \norm{\m_\varep}^{2})+ \varep \norm{ \nabla\cdot \m_\varep}^2 +\frac{\underline \phi}{\Delta t}\norm{\rho_\varep}_Q^2 \\
\le  \norm{\rho_b}_{V'} \left(\norm{\m_\varep}_{L^s} + \norm{\bar f}+\frac{\bar\phi}{\Delta t}\norm{\rho_\varep}_Q\right)  +\norm{\bar f}\norm{\rho_\varep}_{Q}. 
\end{multline*}
This and Young's inequality lead to 
 \beqs
\frac{C}{2} \norm{\m_\varep}_{L^s}^{s}+\frac{\underline \phi}{2\Delta t}\norm{\rho_\varep}_Q^2 
\le  C' \norm{\rho_b}_{V'}^{s^*} + \norm{\rho_b}_{V'}\norm{\bar f}+\left(\frac{\bar\phi}{\Delta t}\norm{\rho_b}_{V'}  +\norm{\bar f}\right)^2\frac{\Delta t}{2\underline \phi}, 
\eeqs
which gives  
\beq\label{m-eps}
\norm{\m_\varep}_{L^s}\le C \left( \norm{\rho_b}_{V'}^{s^*} + \norm{\rho_b}_{V'}^2  +\norm{\bar f}^2\right)^{1/s}, 
\eeq
and 
\beq\label{rho-eps}
\norm{\rho_\varep}_Q\le C\left( \norm{\rho_b}_{V'}^{s^*} + \norm{\rho_b}_{V'}^2  +\norm{\bar f}^2\right)^{1/2}.
\eeq
Plug \eqref{rho-eps} into \eqref{divm-eps} gives
\beq\label{new-divm-eps}
\begin{aligned}
\norm{\nabla\cdot \m_\varep} 
&\le \norm{\bar f}+ C\left( \norm{\rho_b}_{V'}^{s^*} + \norm{\rho_b}_{V'}^2  +\norm{\bar f}^2\right)^{1/2} \\
&\le C \left( \norm{\rho_b}_{V'}^{s^*/2} + \norm{\rho_b}_{V'}  +\norm{\bar f}\right).
\end{aligned}
\eeq 
Thus \eqref{discrete-bound} follows directly from \eqref{m-eps}--\eqref{new-divm-eps}.  
\end{proof}
%
\subsection{Solvability of the semi-discrete problem}
In the same manner as in Section 1 we take limit $\varep\to 0$ and obtain
the existence of a solution of the semi-discrete problem \eqref{w1}. 
\begin{theorem}\label{Sol-semidiscreteProb}
The mixed formulation \eqref{w1} of the semi-discrete problem \eqref{semidiscreteProb} possesses a unique solution $(\m,\rho)\in W(\rm{div};\Omega) \times L^2(\Omega)$.
\end{theorem}
\begin{proof}
Like in the proof of Theorem \ref{SolofStationaryProb} we add both equations in \eqref{eqatstepj} and obtain the nonlinear form a, defined on $(V \times Q) \times (V \times Q)'$, and the linear form $\tilde f \in (V \times Q)'$, defined by
\beqs
a( (\m,\rho),(\vv,q) ):=a(\m,\vv)- b(\vv,\rho)+\intd{ \frac{\phi \rho }{\Delta t},q} +b(\m,q), \quad \tilde f(\vv,q)=- \intb{\rho_b,\vv\cdot\nu} +\intd{\bar f,q}. 
\eeqs
Again, the operator $\mathcal A : V \times Q \to (V \times Q)'$ is defined by $$\intb{\mathcal A(\uu, p), (\vv, q)}_{(V \times Q)'\times (V \times Q)} = a((\uu, p), (\vv, q)).$$ Choosing $\varep = 1/n$, we obtain a sequence of unique solutions $(\m_n, \rho_n)$ of the regularized problems \eqref{reg-eqatstepj}. Owing to Lemma~\ref{boundedness-discrete-sol} the sequence $((\m_n, \rho_n))_{n\in \N}$ is bounded in $V \times Q$. Thus there is a weakly convergent
subsequence, again denoted by $((\m_n, \rho_n))_{n\in \N}$, which converges to $(\m, \rho) \in V \times Q$. In the same manner as in the proof of Theorem~\ref{SolofStationaryProb} we obtain the identity $\mathcal A(\m, \rho) = \tilde f$ in $(V \times Q)'$, i.e., $(\m, \rho)$ is a solution of the semi-discrete mixed formulation \eqref{w1}.
 
To show the uniqueness we consider two solutions $(\m_1, \rho_1)$ and  $(\m_2,\rho_2)$ of \eqref{eqatstepj}. Using the test functions $\vv=\m_1-\m_2,$ and $q=\rho_1-\rho_2$, we obtain 
\begin{align*}
a(\m_1,\m_1-\m_2)-a(\m_2,\m_1-\m_2)  - b(\m_1-\m_2,\rho_1-\rho_2)= 0, \\
\intd{\frac{\phi (\rho_1 -\rho_2)}{\Delta t}, \rho_1-\rho_2} + b(\m_1-\m_2, \rho_1-\rho_2)=0.
\end{align*}
Adding the two equations yield 
\beqs
\begin{split}
0&=a(\m_1,\m_1-\m_2)-a(\m_2,\m_1-\m_2) + \intd{\frac{\phi (\rho_1 -\rho_2)}{\Delta t}, \rho_1-\rho_2}\\
&\ge C_2 \left(\norm{ \m_1-\m_2}^2+\norm{ \m_1-\m_2}_{L^s}^s  \right) + \frac{\underline \phi}{\Delta t} \norm{\rho_1-\rho_2}^2.
\end{split} 
\eeqs 
It follows that $\m_1=\m_2$ and $\rho_1=\rho_2$ a.e. 
\end{proof}
\section {The transient problem}\label{TransProb}
Finally, we address the continuous transient problem. Due to the lack of regularity of the solution $\m$, it is not possible to handle more general boundary conditions as in the previous sections.  We restrict our considerations here to the case of homogeneous Dirichlet boundary conditions
\beq\label{TProb}
\begin{aligned}
\left(\sum_{i=0}^N a_i(x,t)|\m(x,t)|^{\alpha_i}\right) \m(x,t) = -\nabla \rho(x,t), & & \quad (x,t)\in\Omega\times (0,T) , \\
\phi(x)\rho_t(x,t) + \nablad\m(x,t)=f(x,t), & & \quad(x,t)\in\Omega\times (0,T),\\
\rho(x,t) =0, & & \quad (x,t)\in\partial\Omega \times (0,T) ,\\
\rho(x,0)=\rho_0(x), & & \quad x\in\Omega.
\end{aligned}
\eeq
We make the following assumptions
\begin{itemize}
\item [$H_1$.]   $0<\underline{\phi}\le \phi(x) \le \overline\phi<\infty.$
\item [$H_2$.] $f \in L^\infty (0,T;L^2(\Omega)).$
\item [$H_3$.] $\rho_0\in W_0^{1,s^*}(\Omega)\cap L^2(\Omega)$.  
\item [$H_4$.] $a_i(\cdot,t)\in L^\infty(\Omega), i=0,\ldots,N$.
\end{itemize}
Furthermore, we require these coefficient functions to be Lipschitz continuous in time, i.e., there exist a constant $L$ such that for every $0 \le t_1 \le t_2 \le T,$
\beqs
\norm{a_i(t_1)-a_i(t_2)}_{L^\infty}\le L |t_1-t_2|, \quad \text { and }\quad \norm{f(t_1)-f(t_2)}\le L|t_1-t_2|. 
\eeqs

\subsection{A priori estimates for the solutions of the semi-discrete problems}
As mentioned above we use the technique of semi-discretization in time  (see in \cite{raviart69})  to show the existence of solutions of the transient problem~\eqref{TProb}. The existence and uniqueness of the solutions to the semi-discrete problems has been established in Section~\ref{SemiProb}. In the next step, we consider the limit $\Delta t \to 0$. Similar to the regularization technique employed in the last two sections, we derive a priori estimates for the solutions of the semi-discrete problems, which are independent of $\Delta t.$ 

We investigate the semi-discrete problem~\eqref{w1} for homogeneous Dirichlet boundary condition. In this case problem~\eqref{w1} can reads as:   Find $(\m^j, \rho^j)\in W({\rm div}, \Omega) \times L_0^2(\Omega)$, such that
\beq\label{semidiscrete-prob}
\begin{aligned}
a(\m^j, \vv) - b(\vv, \rho^j) = 0, \quad\forall \vv\in V,\\
\intd{\phi\frac{\rho^j -\rho^{j-1}}{\Delta t}, q} + b(\m^j,q) =\intd{f^j, q},  \quad \forall q\in Q. 
\end{aligned}
\eeq
\begin{lemma}\label{boundedness1} For sufficiently small $\Delta t$, there exists constants $\mathcal C_1>0$ and $\mathcal C_2>0$, independent of $\Delta t$ and $J$, such that
\beq\label{indep1}
\norm{\rho^j} \le \mathcal C_1,\, \text{ and } \, \norm{\m^j}+\norm{\m^j}_{L^s}\le \mathcal C_2 \,  \text {  for all } j =1,2,\ldots, J.   
\eeq
\end{lemma}
\begin{proof}
Choosing $(\vv,q) = (\m^j, \rho^j)$ in \eqref{semidiscrete-prob} and adding the resulting equations yields
\beq\label{sum2eqs}
a(\m^j,\m^j)+\intd{\phi\frac{\rho^j -\rho^{j-1}}{\Delta t}, \rho^j} =\intd{ f^j, \rho^j}.
\eeq
Using the identity 
$
2\intd{\phi\frac{\rho^j -\rho^{j-1}}{\Delta t}, \rho^j} =\frac{1}{\Dt}\intd{\phi, |\rho^j|^2 -|\rho^{j-1}|^2+ |\rho^j-\rho^{j-1}|^2},  
$
we get 
\beqs
2\Delta t a(\m^j,\m^j) + \intd{\phi, |\rho^j|^2 -|\rho^{j-1}|^2+ |\rho^j-\rho^{j-1}|^2} = 2\Dt\intd{f^j, \rho^j} . 
\eeqs
Since 
\beqs
2\Delta t a(\m^j,\m^j)\ge 2\Delta tC_2\left( \norm{\m^j}^2+\norm{\m^j}_{L^s}^s\right), \,
\text{
and
}\, 
2\Dt\intd{f^j, \rho^j}
  \le \Dt \left(\norm{f^j}^2 + \norm{\rho^j}^2 \right), 
\eeqs
it follows that
\beqs
2\Delta tC_2\left( \norm{\m^j}^2+\norm{\m^j}_{L^s}^s\right) +\underline{\phi}\left(\norm{\rho^j}^2 -\norm{\rho^{j-1}}^2\right) \le \Dt \left(\norm{f^j}^2+  \norm{\rho^j}^2 \right). 
\eeqs
Using the boundedness of $\phi$ we obtain 
\beqs
\frac{\norm{\rho^j}^2 -\norm{\rho^{j-1}}^2}{\Delta t} -\underline{\phi}^{-1}  \norm{\rho^j}^2 +  2C_2\underline{\phi}^{-1}\left( \norm{\m^j}^2+\norm{\m^j}_{L^s}^s\right)   \le  C\underline{\phi}^{-1}\norm{f^j}^2. 
\eeqs
By discrete Gronwall's inequality 
\begin{multline*}
\norm{\rho^j}^2 + 2C_2\underline{\phi}^{-1}\left( \norm{\m^j}^2+\norm{\m^j}_{L^s}^s\right)\\
\le  C(1- \underline{\phi}^{-1}\Delta t)^{-j}\left( \norm{\rho^0}^2 +  \underline{\phi}^{-1}\Dt \sum_{i=1}^j (1- \underline{\phi}^{-1}\Delta t)^{i-1} \norm{f^i}^2  \right).
\end{multline*} 
Since  $ (1- \ell\Dt)^{-j}\le (1- \ell\Dt)^{-J} \le e^{\frac{J\ell\Dt}{1-\ell\Dt}}=e^{\frac{\ell T}{1-\ell\Dt}}< e^{2\ell T} $ for $\Dt<1/(2\ell)$. It follows from above inequality that
\beq
\norm{\rho^j}^2 + 2C_2\underline{\phi}^{-1}\left( \norm{\m^j}^2+\norm{\m^j}_{L^s}^s\right)  \le  Ce^{c_*T} 
\left( \norm{\rho^0}^2 + T\norm{f}_{L^\infty(0,T;L^2) }^2 \right).
\eeq
 This completes the proof.   
\end{proof}
\begin{lemma}\label{boundedness2} For sufficiently small $\Delta t$, there exists constants $\mathcal C_3>0$  independent of $\Delta t$ and $J$, such that
\beqs
\sum_{i=1}^j \Delta t  \norm{\frac{\rho^j- \rho^{j-1}}{\Delta t} } \le \mathcal C_3,\,  j =1,2,\ldots, J.   
\eeqs
\end{lemma}
\begin{proof}
Choosing the test function $q= \rho^j- \rho^{j-1}$ we obtain from the second equation in \eqref{semidiscrete-prob}
\beq\label{eqs0}
\Delta t\norm{ \sqrt{\phi}\frac{\rho^j -\rho^{j-1}}{\Delta t}  }^2 + b(\m^j, \rho^j- \rho^{j-1}) =\intd{f^j ,\rho^j- \rho^{j-1}}. 
\eeq
Taking $\vv=\m^j$  at time step $j$ and $j-1$ from the first equation in \eqref{semidiscrete-prob} we have,  
\beqs 
a(\m^j, \m^j) - b(\m^j, \rho^j) = 0,\quad  \text{ and }  \quad a(\m^{j-1}, \m^j) - b(\m^j, \rho^{j-1}) =0, 
\eeqs
which implies that
\beq\label{eqs2}
a(\m^j, \m^j)- a(\m^{j-1}, \m^j)= b(\m^j, \rho^j- \rho^{j-1}) .
\eeq
 Substituting \eqref{eqs2} into \eqref{eqs0} and summing up for $j=1, \ldots, J$ yields   
\beq\label{a1}
\sum_{j=1}^J \Dt\norm{ \sqrt{\phi}\frac{\rho^j -\rho^{j-1}}{\Delta t}  }^2 =  \sum_{j=1}^J a(\m^{j-1}, \m^j)-a(\m^j, \m^j) + \sum_{j=1}^J \intd{f^j  ,\rho^j- \rho^{j-1}}.
\eeq
We estimate the right hand side term by term.

The second term on the right hand side of \eqref{a1} are bounded by using H\"older's inequality, \eqref{indep1} and  the Lipschitz of $f$ in the time variable.   
\beq\label{a2}
\begin{aligned}
\sum_{j=1}^J  \intd{f^j ,\rho^j- \rho^{j-1}} &= \intd{f^J ,\rho^J}- \intd{f^1,\rho^0}+\sum_{j=1}^{J-1} \intd{f^j-f^{j-1},\rho^j}\\
&\le \norm{f^J}\norm{\rho^J} + \norm{f^1}\norm{\rho^0} +\sum_{j=1}^{J-1}\norm{f^j-f^{j-1}} \norm{\rho^j}\\
&\le C\left(\mathcal C_1+ \norm{\rho^0}\right)\left(\norm{f}_{L^\infty(0,T; L^2)}+LT\right).  
\end{aligned}
\eeq
For the first term the following estimate holds
\begin{multline}\label{eq0}
 \sum_{j=1}^J a(\m^{j-1}, \m^j)-a(\m^j, \m^j)\\
  = \sum_{j=1}^J \int_\Omega \Big(\big(\sum_{i=0}^N a_i^{j-1}|\m^{j-1}|^{\alpha_i}\big)\m^{j-1} \cdot\m^j - \sum_{i=0}^N a_i^j|\m^j|^{\alpha_i+2}\Big)dx. 
\end{multline}
Applying Young's inequality shows
\begin{multline}\label{intsum}
\sum_{i=0}^N a_i^{j-1}|\m^{j-1}|^{\alpha_i}\m^{j-1} \cdot\m^j \le \sum_{i=0}^N a_i^{j-1}\left( \frac{\alpha_i+1}{\alpha_i+2}|\m^{j-1}|^{\alpha_i+2}+\frac{1}{\alpha_i+2} |\m^j|^{\alpha_i+2}\right)\\
=\sum_{i=0}^N \left(\frac{(\alpha_i+1)a_i^{j-1}}{\alpha_i+2} |\m^{j-1}|^{\alpha_i+2}+ \frac{a_i^{j-1} -a_i^j }{\alpha_i+2}|\m^j|^{\alpha_i+2}+ \frac{a_i^j}{\alpha_i+2}  |\m^j|^{\alpha_i+2}\right).
\end{multline}
Substitute \eqref{intsum} into \eqref{eq0} we obtain 
\beq\label{a3}
\begin{aligned}
& \sum_{j=1}^J a(\m^{j-1}, \m^j)-a(\m^j, \m^j)\\
& \le \sum_{j=1}^J \int_\Omega \Big(\sum_{i=0}^N \frac{\alpha_i+1}{\alpha_i+2} (a_i^{j-1}|\m^{j-1}|^{\alpha_i+2} - a_i^j|\m^j|^{\alpha_i+2}) + \sum_{i=0}^N \frac{a_i^{j-1} -a_i^j}{\alpha_i+2}|\m^j|^{\alpha_i+2}\Big) dx\\
&\quad = \sum_{i=0}^N\int_\Omega \Big( \frac{\alpha_i+1}{\alpha_i+2} (a_i^0|\m^0|^{\alpha_i+2} - a_i^J|\m^J|^{\alpha_i+2}) + \sum_{j=1}^J\frac{a_i^{j-1} -a_i^j}{\alpha_i+2} |\m^j|^{\alpha_i+2}\Big) dx\\
 &\le  \sum_{i=0}^N \Big(\frac{\alpha_i+1}{\alpha_i+2} (\norm{a_i^0}_{L^\infty}\norm{\m^0}_{L^s}^{\alpha_i+2} + \norm{a_i^J}_{L^\infty}\norm{\m^J}_{L^s}^{\alpha_i+2}) + \sum_{j=1}^J \frac{1}{\alpha_i+2} \norm {a_i^{j-1} -a_i^j }_{L^\infty}\norm{\m^j}_{L^s}^{\alpha_i+2}\Big)\\
 &\le  \sum_{i=0}^N \Big( 2\bar a \frac{\alpha_i+1}{\alpha_i+2} +\frac{LT}{\alpha_i+2}\Big) \mathcal C_2^{\alpha_i+2}
   \le 2(\bar a  + LT)(N+1)(1+\mathcal C_2)^{\alpha_N+2}.
\end{aligned}
\eeq
It follows from \eqref{a1}--\eqref{a3} that 
\begin{multline*}
\underline \phi \sum_{j=1}^J \Dt\norm{\frac{\rho^j -\rho^{j-1}}{\Dt}}^2 \le \sum_{j=1}^J \Delta t\norm{\sqrt{\phi}\frac{\rho^j -\rho^{j-1}}{\Delta t}}^2 \\
\le C\left(\mathcal C_1+ \norm{\rho^0}\right)\left(\norm{f}_{L^\infty(0,T; L^2)}+LT\right)+2(\bar a  + LT)(N+1)(1+\mathcal C_2)^{\alpha_N+2}\eqdef \mathcal C_3. 
\end{multline*}
This completes the proof. 
\end{proof}

Next, we show that the mixed formulation \eqref{semidiscrete-prob} is equivalent to a variational formulation of the time-discretized parabolic equation.  To this end, we recall the nonlinear mapping $K$ of \eqref{rua}.  For fixed time $t=t_j$, we define the nonlinear mapping $K^j: \Omega \times \R^+ \to \R^+$ (see in \eqref{Kdef}) and its inverse defined by  
\beq\label{eq2j}
F^j(x,z)=a_0(x,t_j)z^{\alpha_0} + a_1(x,t_j)z^{\alpha_1}+\cdots +a_N(x,t_j)z^{\alpha_N},\quad z\ge 0.
\eeq 
\begin{lemma}\label{equivProbs}

i) If $\rho^j \in R(\Omega)=\{r\in L_0^2(\Omega), \nabla r\in (L^{s^*}(\Omega))^d \}$ is a solution of the variational formulation. Find $\rho^j\in R(\Omega)$ such that
\beq\label{VarProb}
\intd{\phi\frac{\rho^j-\rho^{j-1} }{\Delta t},  q} + \intd{ K^j(x,|\nabla \rho^j|)\nabla \rho^j, \nabla q } =\intd{ f^j,q}, \quad \forall q\in R(\Omega) 
\eeq
 then $(-K^j(|\nabla \rho^j|)\nabla \rho^j, \rho^j)  $ is a solution of the mixed formulation~\eqref{semidiscrete-prob}. 

ii) If $(\m^j, \rho^j)\in W(\rm{div},\Omega)\times L_0^2(\Omega)$ is a solution of the mixed formulation~\eqref{semidiscrete-prob} then
$\rho^j$ is a solution of the variational formulation \eqref{VarProb}.  In particular, $\rho^j \in R(\Omega) $.
\end{lemma}
\begin{proof} i) Let $\rho^j$ be a solution of \eqref{VarProb}. We define $\m^j=-K^j(x,|\nabla \rho^j|)\nabla \rho^j$. Then Green's formula yields
\beqs
\intd{F^j(x,|\m^j|)\m^j , \vv }  =-\intd{\nabla \rho^j, \vv} =\intd{\rho^j, \nabla\cdot \vv},  \quad \forall \vv\in V.
\eeqs
This is the first equation in \eqref{semidiscrete-prob}. To derive the second equation in~\eqref{semidiscrete-prob}, we consider \eqref{VarProb} for
$q\in \mathcal D(\Omega)\subset R(\Omega)$ 
\beqs
\intd{\phi\frac{\rho^j-\rho^{j-1}}{\Delta t}, q} -(\m^j, \nabla q)=  \intd{f^j, q }, 
\eeqs
and then apply Green's formula we obtain  
\beqs
\intd{\phi\frac{\rho^j-\rho^{j-1}}{\Delta t}, q}+ \intd{ \nabla \cdot \m^j, q } =\intd{f^j, q}.   
\eeqs

Because  $\mathcal D(\Omega)$ is densely embedded into $L^2(\Omega)$, the second equation in \eqref{semidiscrete-prob} follows.

ii) Let $(\m^j,\rho^j)$ be the solution of \eqref{semidiscrete-prob}. Applying Green's formula implies
\beqs
\intd{ F^j(x,|\m^j|)\m^j,  \vv}= \intd{ \nabla \cdot \vv , \rho^j } = \intd{ -\nabla \rho^j, \vv },\quad \forall \vv\in (\mathcal D(\Omega))^d.   
\eeqs
Thus in the sense of distributions it holds $\nabla \rho^j= -F^j(x,|\m^j|)\m^j \in (L^{s*}(\Omega))^d $. Consequently, $\rho^j\in \{r\in L^2(\Omega), \nabla r\in (L^{s^*}(\Omega))^d \}$  and $\m^j=-K^j(|\nabla \rho^j|)\nabla\rho^j$. To prove that $\rho^j$ fulfills \eqref{VarProb}, we consider $q\in R(\Omega)\subset L^2(\Omega) $ in the first equation of \eqref{semidiscrete-prob}. Using integration by parts, we have  
 \beqs
 \begin{split}
 \intd{f^j,q}&= \intd{\phi\frac{\rho^j -\rho^{j-1}}{\Delta t}, q } + \intd{\nabla\cdot \m^j, q } \\
 &= \intd{\phi\frac{\rho^j -\rho^{j-1}}{\Delta t}, q} - \intd{ \nabla\cdot ( K^j(x,|\nabla \rho^j|)\nabla \rho^j , q}\\
 &= \intd{\phi\frac{\rho^j -\rho^{j-1}}{\Delta t}, q } + \intd{ K^j(x, |\nabla \rho^j|)\nabla \rho^j , \nabla q }.
 \end{split}
 \eeqs
 Finally, we consider again the first equation of \eqref{semidiscrete-prob} for $\vv\in(\mathcal D(\bar \Omega))^d$. Using integration by parts, we obtain
 \beqs
 0=-\intd{F^j(x,|\m^j|)\m^j ,\vv}+ \intd{\nabla \cdot \vv ,\rho^j } = \intd{ \nabla\rho^j ,\vv}+ \intd{\nabla \cdot \vv, \rho^j } = \int_{\partial\Omega} \gamma_0\rho^j \vv\cdot\nu d\sigma. 
 \eeqs
 Consequently, $\gamma_0\rho^j=0$ in $W^{1/2, 2}(\partial\Omega)$. i.e. $\rho^j\in R(\Omega)$.  
  \end{proof}
  
Using this equivalence, we obtain a bound for $\rho^j$ in the norm of $R(\Omega)$ defined by $\norm{r}_R = \norm{r}+ \norm{\nabla r}_{L^{s^*}}$.
\begin{lemma}\label{boundedness3} For sufficiently small $\Delta t$, there exist constants $\mathscr C_1,\mathscr C_2$ and $\mathscr C_3$, all independent of $\Delta t$ and $J$, such that
\begin{align}
\label{indep2} \norm{\rho^j}_R&\le \mathscr C_1,\, \text{for all } j=0,1,2,\ldots, J; \\ 
\label{indep3} \norm{\frac{\rho^j -\rho^{j-1} }{\Delta t} }_{R'} &\le \mathscr C_2, \, \text{for all } j=0,1,2,\ldots, J;\\
\label{indep4} \norm{\nabla \cdot \m^j}_{R'}&\le \mathscr C_3,\, \text{for all } j=0,1,2,\ldots, J.
\end{align}

\end{lemma}
\begin{proof}
By H\"older inequality and \eqref{indep1} imply 
\begin{align*}
\norm{\nabla \rho^j}_{L^{s^*}}^{s^*} &=\norm{F^j(x,|\m^j|)\m^j}_{L^{s^*}}^{s^*} \le C \sum_{i=0}^N\int_\Omega |\m^j|^{(\alpha_i+1)s^*} dx\\
&\le C\sum_{i=0}^N\norm{\m^j}_{L^s}^{(\alpha_i+1)s^*} \le  C\sum_{i=0}^N\mathcal C_2^{(\alpha_i+1)s^*} :=\mathscr C_1^{s^*}.
\end{align*} 
This and \eqref{indep1} show \eqref{indep2}. 

By means of \eqref{VarProb}, we have for $q\in R(\Omega)$,
\beq\label{diffquotion}
\begin{split}
\left |\intd{\phi\frac{\rho^j-\rho^{j-1}}{\Delta t},q } \right |&=\left| \intd{ f^j, q} -\intd{ K^j(|\nabla \rho^j|)\nabla \rho^j,q } \right|=\left| \intd{f^j,q} + \intd{\m^j,\nabla q}\right|\\
  &\le \norm{f^j}\norm{q} + \norm{\m^j}_{L^s}\norm{\nabla q}_{L^{s^*}}\le \left(\norm{f^j} + \norm{\m^j}_{L^s}\right)\norm{q}_R.
\end{split}
\eeq
Then \eqref{indep3} follows \eqref{diffquotion} and the boundedness of the function $\phi$. 

From the second equation of \eqref{semidiscrete-prob} yields
\beq\label{nablaM}
\begin{split}
\left|\intd{\nabla\cdot \m^j, q }\right| 
 &= \left |\intd{ f^j, q }- \intd{\phi\frac{\rho^j-\rho^{j-1}}{\Delta t}, q } \right | \le \left(\norm{f^j}+ \bar \phi \norm{\frac{\rho^j-\rho^{j-1}}{\Delta t} } \right) \norm{q}\\
 &\le \left(\norm{f^j}+ \bar \phi \norm{\frac{\rho^j-\rho^{j-1}}{\Delta t} }   \right) \norm{q}_R. 
\end{split}
\eeq
Hence inequality \eqref{indep4} follows from combining \eqref{nablaM} and \eqref{indep3}.   
\end{proof}
\subsection{Solvability of the continuous problem} 
Due to the existence of unique solutions to the semi-discrete mixed formulation \eqref{semidiscrete-prob} we obtain for every $J\in \N$ a $J+1$-tuple of solutions $(\m^j, \rho^j)_{j=0,\ldots,J}\in ( W(\rm{div}, \Omega) \times L_0^2(\Omega) ) ^{J+1}. $ We denote these $J + 1$-tuples with $\m_{\Dt}:=(\m^j)_{j=0,\ldots, J}\in  ( W(\rm{div}, \Omega) ) ^{J+1} $ and $ \rho_{\Dt}:= (\rho^j)_{j=0,\ldots, J}\in  ( L_0^2(\Omega) ) ^{J+1}. $ We define step function by 
\beqs
\pi\rho_{\Dt}(t)=\begin{cases}
\rho^0 & \text { if } t=0 \\
 \rho^j & \text { if } t_{j-1}<t \le  t_j, j=1,\ldots,J
 \end{cases} 
 \in L^\infty(0,T; R(\Omega))
\eeqs
and piecewise linear (in time) functions 
$\Pi\rho_{\Dt}(t)=\frac{\rho^j-\rho^{j-1}}{\Dt} (t-t_j) + \rho^j, t_{j-1}\le t\le t_j, j=1, \ldots,J.$
The time derivative of $\Pi \rho_{\Dt}(t)$ is a piecewise constant step function with values
 \beqs
  \frac{\partial \Pi\rho_{\Dt}}{\partial t} = \frac{\rho^j-\rho^{j-1}}{\Delta t} \quad  \text{ if }  t_{j-1}< t< t_j, j=1,\ldots, J. 
 \eeqs
In addition, we use piecewise constant approximations ${a_i}_{\Dt}$ and $f_{\Dt}$ of the coefficient functions $a_i$ and $f$, and piecewise constant operators $F_{\Dt}$ and $K_{\Dt}$. According to Lemmas~\ref{boundedness1} and \ref{boundedness3} the following bounds hold for sufficiently small $\Delta t$.
\begin{align*}
&\norm{\pi\rho_{\Dt}}_{L^\infty(0,T; R(\Omega))}\le \mathscr C_1, && \norm{\frac{\partial \Pi\rho_{\Dt}}{\partial t} }_{L^\infty(0,T;R'(\Omega))}\le \mathscr C_2,\\
&\norm{\pi\m_{\Dt}}_{L^\infty(0,T; (L^2(\Omega))^d)}\le \mathcal C_2, &&\norm{\pi\nabla\cdot\m_{\Dt}}_{L^\infty(0,T;R'(\Omega))}\le \mathscr C_3,\\
&\norm{ F(x,|\pi\m_{\Dt}|)\pi\m_{\Dt}}_{L^\infty(0,T; (L^{s^*}(\Omega))^d)}\le \mathscr C_1, &&\norm{\rho^J}\le \mathcal C_1.
\end{align*}

Thus the exist a subsequences, again indexed by $\Dt$, that converge in corresponding weak*-topology; in detail 
 \begin{align}
\label{b1} &\pi\rho_{\Dt} \stackrel{*}{\rightharpoonup} \rho \, \text { in }\, L^\infty(0,T; R(\Omega)), 
&& \frac{\partial \Pi\rho_{\Dt}}{\partial t} \stackrel{ *}{\rightharpoonup} \rho' \, \text { in }\, L^\infty(0,T;R'(\Omega))\\
\label{b2} & \pi\m_{\Dt}\stackrel{ *}{\rightharpoonup}  \m \, \text { in } \, L^\infty(0,T; (L^2(\Omega))^d),
&&\pi\nabla\cdot\m_{\Dt} 	\stackrel{ *}{\rightharpoonup}  \bar\m \, \text { in } \, L^\infty(0,T;R'(\Omega))\\
\label{b3} &F(x,|\pi\m_{\Dt}|)\pi\m_{\Dt}	\stackrel{ *}{\rightharpoonup}  \hat F\, \text { in }\, L^\infty(0,T; (L^{s^*}(\Omega))^d), 
&&\rho^J\stackrel{ *}{\rightharpoonup}  \rho_T \,\text { in }\quad L^2(\Omega).
\end{align}
\begin{lemma}\label{identeq0}
i) The identity $\rho'=\partial \rho/\partial t$ hold in the sense of distribution from $(0,T)$ to $L^2(\Omega)$. That is  for all $\varphi\in \mathcal D ((0,T))$ 
\beqs
\int_0^T \rho'(t)\varphi (t) dt = - \int_0^T \rho(t)\varphi'(t) dt \, \text{ in } L^2(\Omega).  
\eeqs
ii) The identity $\bar \m= \nabla\cdot \m$ hold in the sense of distribution on $\Omega$ hold for almost everywhere in   $(0,T)$. That is  for all $\psi\in \mathcal D (\Omega)$ 
\beq\label{coincide2}
\intd{\bar \m, \psi} = - \intd{\m, \nabla \psi }, \, \, \text{ a.e in~}(0,T).  
\eeq
iii) The identity $\hat F = -\nabla \rho$ hold in $L^\infty (0,T; (L^{s^*}(\Omega))^d ) $. That is  for all $\vv\in L^1(0,T; (L^{s}(\Omega))^d )$ 
\beq\label{coincide3}
\int_0^T \intd{\hat F, \vv} dt = - \int_0^T\intd{\nabla \rho,  \vv}  dt.  
\eeq
\end{lemma}
\begin{proof}
i) 
Let $\varphi\in \mathcal D (0,T)$ then  
\begin{align*}
\int_0^T \rho'(t) \varphi (t) dt &= \lim_{\Dt\to 0} \int_0^T \frac{\partial\Pi\rho_{\Dt}}{\Dt}  \varphi (t) dt
 = \lim_{\Dt\to 0} \sum_{i=j}^J \int_{t_{j-1}}^{t_j}   \frac{\partial \Pi\rho_{\Dt}}{\partial t} \varphi (t) dt\\
&=\lim_{\Dt\to 0} \sum_{j=1}^J \Pi\rho_{\Dt}(t_j) \varphi (t_j) - \Pi\rho_{\Dt}(t_{j-1})  \varphi(t_{j-1})   - \int_{t_{j-1}}^{t_j}  \Pi\rho_{\Dt}  \varphi' (t) dt\\
&=-\lim_{\Dt\to 0}\sum_{j=1}^J \int_{t_{j-1}}^{t_j}  \Pi\rho_{\Dt}  \varphi' (t) dt= -\lim_{\Dt\to 0}\int_0^T  \Pi\rho_{\Dt}  \varphi' (t) dt\\
&=-\lim_{\Dt\to 0}\int_0^T  \pi\rho_{\Dt}  \varphi' (t) dt=-\int_0^T \rho \varphi' (t) dt.
\end{align*}

ii) 
Let $\psi\in \mathcal D(\Omega)$ and $\varphi\in \mathcal D(0,T)$ then
\begin{align*}
 \int_0^T \intd{\bar \m, \psi} \varphi (t) dt &= \lim_{\Dt\to 0}\int_0^T \intd{\pi\nabla\cdot\m_{\Dt},\psi}\varphi(t) dt\\
 &  = - \lim_{\Dt\to 0}\int_0^T \intd{ \pi\m_{\Dt},\nabla \psi}\varphi(t) dt= - \int_0^T (\m,\nabla\psi) \varphi(t) dt.
\end{align*}

iii) 
For all $\vv\in L^1(0,T,(L^s(\Omega))^d)$, 
\begin{align*}
\int_0^T\intd{ \hat F, \vv}dt &=\lim_{\Dt\to 0}\int_0^T \intd{  F(x,|\pi\m_{\Dt}|)\pi\m_{\Dt}, \vv}dt\\
& = -\lim_{\Dt\to 0}\int_0^T \intd{  \pi\nabla \rho_{\Dt}, \vv }dt =- \int_0^T\intd{\nabla \rho, \vv}dt.
\end{align*}
\end{proof}
\begin{lemma}\label{identeq1}
The following identity holds in $L^\infty(0,T; L_0^2(\Omega))$
\beq\label{firsteq}
\phi\frac{\partial \rho}{\partial t} +\nabla\cdot \m = f. 
\eeq
Furthermore $\rho(x,0)=\rho_0(x)$, and $\rho(x,T)=\rho_T$.
\end{lemma}
\begin{proof}
For $\varphi\in \mathcal D ([0,T])$ we defined the step function $\varphi_{\Dt}$ by 
\beqs
\varphi_{\Dt} =\begin {cases}  
                        \varphi(t_{j-1}) &\text { if } t_{j-1}\le t< t_j, j=1,\ldots, J\\
                        \varphi(T)& \text { if }         t=T          
                      \end{cases}.
\eeqs
Using the test function $q=\psi\in\mathcal D(\Omega)$ in the second equation  of \eqref{semidiscrete-prob}, multiplying by $\Dt\varphi(t_{j-1})$ and summing up on $j=1,\ldots,J$, we obtain 
\beq\label{seo}
\sum_{j=1}^J \intd{ \phi\frac{\rho^j -\rho^{j-1}}{\Delta t}, \psi}  \Dt\varphi(t_{j-1})  + \intd{\nabla\cdot\m^j, \psi} \Dt\varphi(t_{j-1})  =\sum_{j=1}^J\intd {f^j, \psi} \Dt\varphi(t_{j-1}) 
\eeq 
Using the piecewise constant function $\pi$ this reads
  \beqs
\int_0^T\intd{ \phi \frac{\partial \Pi\rho_{\Dt}}{\partial t}, \psi} \varphi_{\Dt} dt  + \int_0^T\intd{ \pi \nabla\cdot\m_{\Dt}, \psi} \varphi_{\Dt}  dt =\int_0^T\intd{ \pi f, \psi} \varphi_{\Dt}  dt.
\eeqs 
Since $\psi\varphi_{\Dt}$ converges strongly to $\psi\varphi$ in $L^1(0,T;R(\Omega))$. Hence limit  $\Dt\to 0$
\beq\label{seo2}
\int_0^T\intd{ \phi \frac{\partial\rho}{\partial t}, \psi} \varphi  dt  + \int_0^T\intd{ \nabla\cdot\m, \psi} \varphi dt =\int_0^T\intd{ f, \psi} \varphi dt.  
\eeq 
The set $\{  \psi\varphi, \psi\in \mathcal D(\Omega), \varphi\in \mathcal D ( \overline{(0,T)} ) \}$ is dense subset of $L^1(0,T; R(\Omega))$. Thus the identity \eqref{firsteq}  is established. 
 
 To prove the remaining two identities we rewrite \eqref{seo} as form
  \begin{multline*}
- \sum_{j=1}^J \Dt\intd{\phi\rho^j, \psi}\frac{\varphi(t_j)- \varphi(t_{j-1})}{\Dt} + \sum_{j=1}^J \Dt\intd{ \nabla\cdot\m^j, \psi} \varphi(t_{j-1}) \\
=\sum_{j=1}^J \Dt\intd{ f^j, \psi }\varphi(t_{j-1})
-  \intd{ \phi\rho^J, \psi}\varphi(T)+\intd{\phi\rho^0, \psi}\varphi(0),
\end{multline*}
which is 
 \begin{multline*}
  -\int_0^T\intd {\phi \pi\rho_{\Dt}, \psi} \frac{\partial \Pi\varphi}{\partial t} dt  + \int_0^T\intd{ \pi \nabla\cdot\m_{\Dt}, \psi} \varphi_{\Dt} dt\\
   =\int_0^T\intd{\pi f, \psi} \varphi_{\Dt} dt-  \intd{\phi\rho^J,\psi}\varphi(T)+\intd{\phi\rho^0,\psi}\varphi(0).
 \end{multline*}
   Passing to the limit $\Dt\to 0$ we obtain
   \beq\label{ab0}
   -\int_0^T\intd{\phi \rho, \psi} \frac{\partial \varphi}{\partial t} dt  + \int_0^T\intd{ \nabla\cdot\m, \psi} \varphi dt =\int_0^T\intd{ f, \psi }\varphi dt-  \intd{ \phi\rho_T, \psi}\varphi(T)+\intd{\phi\rho^0, \psi}\varphi(0).
   \eeq
   In the other hand, partial integration of \eqref{seo2} yields
   \beq\label{ab1}
-\int_0^T\intd{\phi \rho,\psi} \frac{\partial \varphi }{\partial t} dt  + \int_0^T\intd{\nabla\cdot\m, \psi} \varphi dt =\int_0^T\intd{ f, \psi} \varphi dt- \intd{\phi \rho(T),\psi} \varphi(T)+ \intd{ \phi \rho(0),\psi} \varphi(0).  
\eeq 
 It follows from \eqref{ab0} and \eqref{ab1} that
 \beqs
 \intd{\phi (\rho(0)-\rho^0), \psi} \varphi(0)=\intd{\phi (\rho(T)-\rho_T), \psi} \varphi(T).
   \eeqs
   Since $\varphi(0)$ and $\varphi(T)$ are arbitrary 
\beqs
\intd{\phi (\rho(0)-\rho^0), \psi}=0=\intd{\phi (\rho(T)-\rho_T), \psi}
   \eeqs
Thus $\rho(0)=\rho^0$ and $\rho(T)=\rho_T$.
  \end{proof}
  
\begin{lemma}\label{Fhat-Eq-F} The limit $\m$ of $\pi\m_{\Dt}$ and $\hat F$ of $ F(x,t, \pi\m_{\Dt})$ satisfy 
$\hat F = F(x,t,|\m|)\m$ in $L^\infty(0,T; (L^{s^*}(\Omega))^d )$. That means  
\beqs
\int_0^T \intd{ \hat F,  \vv } dt = \int_0^T  \intd{ F(x,t,|\m|)\m, \vv }dt, 
\eeqs
for all $\vv\in L^1(0,T; (L^s(\Omega))^d)$. 
\end{lemma}  

To show this, we need an auxiliary result, a particular result of Lemma 1.2 in \cite{raviart69}.  
\begin{proposition}\label{aux1}
The limit $\rho$ of $\pi \rho_{\Dt}$ satisfy 
\beq
\int_0^T  \left(\phi \frac{\partial \rho (t) }{\partial t}, \rho(t)\right) dt =\frac 1{2}   \int_\Omega\phi |\rho(T)|^2   - \phi |\rho(0)|^2  dx. 
\eeq
\end{proposition}
\begin{proof}
Equation \eqref{sum2eqs} rewrite as
\beqs
\intd{F(x,t,|\m^j|)\m^j , \m^j }+\intd{\phi\frac{\rho^j -\rho^{j-1}}{\Delta t}, \rho^j } =\intd{f^j, \rho^j }.
\eeqs
Since $ \intd{\phi\frac{\rho^j -\rho^{j-1}}{\Delta t}, \rho^j} \ge \frac 1 {2\Dt} \intd{ \phi,|\rho^j|^2 -|\rho^{j-1}|^2 } $,
\beqs
\intd{ F(x,t,|\m^j|)\m^j, \m^j}+\frac 1 {2\Dt} \intd{\phi,|\rho^j|^2 -|\rho^{j-1}|^2} \le \intd{f^j, \rho^j }.
\eeqs
Multiplying $\Dt$ and summing up $j=1,\ldots, J$, we obtain  
\beqs
\sum_{i=1}^J  \int_{t_{i-1}}^{t_i} \intd{ F(x,t,|\m^j|)\m^j , \m^j}+\frac 1 {2} \intd{\phi,|\rho^J|^2 -|\rho^{0}|^2} \le \sum_{i=1}^J  \int_{t_{i-1}}^{t_i} \intd{ f^j, \rho^j }dt.
\eeqs
Then we rewrite as
\beqs
\sum_{i=1}^J  \int_{t_{i-1}}^{t_i} \intd{ F(x,t,|\pi\m_{\Dt}|)\pi\m_{\Dt}, \pi\m_{\Dt}  }+\frac 1 {2} \intd{\phi,|\rho^J|^2} -\intd{\phi,|\rho^{0}|^2} \le \sum_{i=1}^J  \int_{t_{i-1}}^{t_i} \intd {\pi f_{\Dt}, \pi\rho_{\Dt} } dt,
\eeqs
which is  
\beqs
 \int_{0}^{T} \intd{ F(x,t,|\pi\m_{\Dt}|)\pi\m_{\Dt},\pi \m_{\Dt} }+\frac 1 {2} \intd{\phi,|\rho^J|^2} -\intd{\phi,|\rho^{0}|^2} \le \int_{0}^{T} \intd{ \pi f_{\Dt}, \pi\rho_{\Dt} } dt.
\eeqs
Take the limit inferior we conclude that
\beqs
 \liminf_{\Dt\to 0}\int_{0}^{T} \int_{0}^{T} \intd{  F(x,t,|\pi\m_{\Dt}|)\pi\m_{\Dt},\pi \m_{\Dt} }+\frac 1 {2} \intd{\phi,|\rho^J|^2} -\intd{\phi,|\rho^{0}|^2} \le \int_{0}^{T} \intd{ f, \rho } dt.
\eeqs
Using the result of Proposition~\ref{aux1}, we find that 
\beqs
 \liminf_{\Dt\to 0}\int_{0}^{T} \intd{ F(x,t,|\pi\m_{\Dt}|)\pi\m_{\Dt},\pi \m_{\Dt} }+  \int_0^T \intd{\phi\frac{\partial \rho}{\partial t}, \rho } dt   \le \int_{0}^{T} \intd{f, \rho} dt.
\eeqs
On the other hand, Lemma~\ref{identeq1} gives  
\beqs
\int_0^T\intd{\phi\frac{\partial \rho}{\partial t}, \rho} dt +\int_0^T \intd{\nabla\cdot \m, \rho} dt = \int_0^T \intd{f, \rho} dt. 
\eeqs
From \eqref{coincide2} and \eqref{coincide3} in Lemma~\ref{identeq0} we have 
\beqs
\int_0^T \intd{\nabla \cdot\m, \rho} dt = -\int_0^T \intd{ \m,\nabla\rho} dt =\int_0^T \intd{\m,\hat F}dt.
\eeqs
 Therefore, 
 \beqs
  \liminf_{\Dt\to 0}\int_{0}^{T} \int_{0}^{T} \intd{ F(x,t,|\pi\m_{\Dt}|)\pi\m_{\Dt},\pi \m_{\Dt} }\le \int_0^T \intd{\m,\hat F}dt.
\eeqs
We have shown that for arbitrary $\vv\in L^\infty(0,T; (L^s(\Omega))^d )$
\beqs
\begin{aligned}
&\int_0^T\intd{\hat F - F(x,t,|\vv|\vv), \m-\vv} dt \\
&\qquad \ge \liminf_{\Dt\to 0} \int_0^T\intd{ F(x,t,|\pi\m_{\Dt}|)\pi\m_{\Dt} - F(x,t,|\vv|)\vv , \pi \m_{\Dt} -\vv } dt \ge 0. 
\end{aligned}
\eeqs
Choose $\vv= \m -\lambda  \varphi$, $\lambda>0$, $\varphi\in L^\infty(0,T; (L^s(\Omega))^d ) $ 
\beqs
\int_0^T\intd{\hat F - F(x,t,|\m -\lambda  \varphi|)(\m -\lambda  \varphi), \lambda\varphi} dt \ge 0. 
\eeqs
Dividing $\lambda$ and letting $\lambda\to 0$, we obtain 
\beqs
\int_0^T\intd{\hat F - F(x,t,|\m|)\m, \varphi} dt \ge 0 \quad \forall \varphi\in L^\infty(0,T; (L^s(\Omega))^d ). 
\eeqs
This implies $\hat F = F(x,t,|\m|)\m$.  (cf. the proof of Thm. 1.1 in \cite{raviart69} page 313).
\end{proof}

\begin{theorem}\label{ExistenceSol}
For all $f \in L^\infty(0, T; L^2(\Omega)) $ that are Lipschitz continuous in time $t$. There exists a pair $(\m,\rho) \in L^\infty(0,T; W({\rm div}, \Omega))\times L^\infty(0,T; L_0^2),$ such that
\begin{align*}
\int_0^T \intd{F(x, |\m|)\m, \vv} dt - \int_0^T \intd{\nabla \cdot \vv, \rho} =0,\,\,& \text{ for all } \vv \in L^1(0,T; L^2(\Omega)),\\
\int_0^T \intd{\phi\frac{\partial \rho}{\partial t}, q}dt +\int_0^T \intd{\nabla\cdot \m, q} dt=\int_0^T\intd{f,q} dt, \,\, & \text{ for all } q \in L^1(0,T; L_0^2(\Omega)).
\end{align*}
\end{theorem}
\begin{proof}
Let $\m$ be the limit of $\pi\m_{\Dt}$ and $\rho$ be the limit of $\pi\rho_{\Dt}$. Then Lemma~\ref{Fhat-Eq-F} and Lemma~\ref{identeq0} part iii) imply that
\beqs
\int_0^T \intd{F(x, t, |\m|)\m,\vv } dt = \int_0^T \intd{\hat F,\vv } dt =- \int_0^T \intd{\nabla \rho,\vv } dt= \int_0^T \intd{\vv, \nabla\cdot \rho }dt,
\eeqs
for all $ \vv \in L^1(0,T; L^2(\Omega))$. In Lemma~\ref{identeq1} we have seen
that $(\m, \rho)$ fulfills the second equation.
\end{proof}

\myclearpage
\myclearpage
\appendix

\def\cprime{$'$} \def\cprime{$'$} \def\cprime{$'$}

\end{document}